\documentclass[a4paper,12pt]{amsart}
\usepackage[utf8]{inputenc}
\usepackage{hyperref}
\usepackage{amsmath, amssymb}
\usepackage[alphabetic,nobysame, initials]{amsrefs}
\usepackage{graphicx}

\newtheorem{thm}{Theorem}[section]
\newtheorem{lem}[thm]{Lemma}

\theoremstyle{definition}

\newtheorem{rem}[thm]{Remark}

\newcommand{\Ze}{\mathbb Z}
\renewcommand{\Re}{\mathbb R}
\newcommand{\Ren}{\Re^n}
\newcommand{\Renp}{\Re^{n+1}}
\newcommand{\Ben}{{\mathbf B}^n}
\newcommand{\Benp}{{\mathbf B}^{n+1}}

\newcommand{\Se}{\mathbb S}
\newcommand{\Sen}{\Se^{n-1}}
\newcommand{\Senp}{\Se^{n}}
\newcommand{\Omegan}{\Omega_{n-1}}
\newcommand{\Omeganp}{\Omega_{n}}

\renewcommand{\epsilon}{\varepsilon}

\newcommand{\Pe}{\mathbb P}

\newcommand{\st}{\; : \; }
\renewcommand{\phi}{\varphi}
\DeclareMathOperator{\inter}{int}

\DeclareMathOperator{\bd}{bd}

\DeclareMathOperator{\conv}{conv}
\newcommand{\card}[1]{\left|#1\right|}
\newcommand{\Binom}{\mathtt{Binom}}
\newcommand\fact[2][]{
  \ifx&#1& 
  \refstepcounter{equation}
  \fi
  \begin{minipage}{0.09\textwidth}
  \ifx&#1& 
  (\theequation)
  \else 
  {#1}
  \fi
  \end{minipage}
  \hfill
  \begin{minipage}{0.81\textwidth}
  \emph{
  \begin{sloppypar}
  #2
  \end{sloppypar}
  }
  \end{minipage}
  \smallskip
}

\title{A Spiky Ball}
\author[M. Nasz\'odi]{M\'arton Nasz\'odi}
\address{
ELTE, 
Dept. of Geometry,
Lorand E\"otv\"os University,
P\'azm\'any P\'eter S\'et\'any 1/C
Budapest, Hungary 1117
}
\email{marton.naszodi@math.elte.hu}
\thanks{The author acknowledges the support of the J\'anos Bolyai 
Research Scholarship of the Hungarian Academy of Sciences, and
the Hung. Nat. Sci. Found. (OTKA) grant PD104744.}

\begin{document}
\begin{abstract}
The Illumination Problem may be phrased as the problem of covering a convex 
body in Euclidean $n$-space by a minimum number of translates of its interior. 
By a probabilistic argument, we show that, arbitrarily close to the Euclidean 
ball, there is a centrally symmetric convex body of illumination number 
exponentially large in the dimension.
\end{abstract}

\maketitle

\section{Introduction}

For two sets $K$ and $L$ in $\Ren$, let $N(K,L)$ denote the 
\emph{translative covering number} of $K$ by $L$, that is, the
minimum number of translates of $L$ that cover $K$.

Let $K$ be a convex body (that is, a compact, convex set with non--empty 
interior) in $\Ren$. 
Following Hadwiger \cite{Ha60}, we say that a point $p\in\Ren\setminus K$ 
\emph{illuminates} a boundary point 
$b\in\bd K$, if the ray $\{p+\lambda (b-p)\st \lambda>0\}$
emanating from $p$ and passing through $b$ intersects the interior of $K$.
Boltyanski \cite{Bo60} gave the following slightly different definition. A 
direction $u\in\Se^{n-1}$ is said to 
\emph{illuminate} $K$ at a boundary point $b\in\bd K$ if the ray $\{b+\lambda 
u\st \lambda>0\}$ intersects the interior of $K$.
It is easy to see that the minimum number of directions 
that illuminate each boundary point of $K$ is equal to the minimum number of 
points that illuminate each boundary point of $K$. This number is called the 
\emph{illumination number} ${\rm i}(K)$ of $K$. 

We call a set of the form $\lambda K+v$ a \emph{smaller positive homothet of} 
$K$ if $0<\lambda<1$ and $v\in\Ren$.
Gohberg and Markus asked how large the minimum number of smaller positive 
homothets of $K$ covering $K$ can be. It is not hard to see that this number is 
equal to $N(K,\inter K)$. It is also easy to see that ${\rm i}(K)=N(K,\inter 
K)$.

Any smooth convex body (ie., a convex body with a unique support hyperplane at 
each boundary point) in $\Ren$ is illuminated by $n+1$ 
directions. Indeed, for a smooth body, the set of directions illuminating a 
given boundary point is an open hemisphere of $\Se^{n-1}$, and one can find 
$n+1$ points (eg., take the vertices of a regular simplex) in $\Se^{n-1}$ with 
the property that every open hemisphere contains at least one of the points. 
Thus, these $n+1$ points in $\Se^{n-1}$ (ie., directions) illuminate any smooth 
convex body in $\Ren$ (cf. \cite{BoMaSo97} for details).

On the other hand, the illumination number of the cube is $2^n$, since no two 
vertices of the cube share an illumination direction. Even though we do not 
discuss it, it would be difficult to omit mentioning the 
\emph{Gohberg--Markus--Levi--Boltyanski--Hadwiger Conjecture} (or, Illumination 
Conjecture), according to which for any convex body $K$ in $\Ren$, we have ${\rm 
i}(K)=2^n$, where equality is attained only when $K$ is an affine image of the 
cube.

For more background on the problem of illumination, see \cites{Be06, 
Be10, BMP05, MS99}. In Chapter VI. of \cite{BoMaSo97}, one can find a proof of 
the equivalence of the four definitions of ${\rm i}(K)$ given above.

The Euclidean ball is a smooth convex body, and hence, is of illumination 
number $n+1$. Theorem~\ref{thm:spiky} shows that, arbitrarily close to the 
Euclidean ball, there is a convex body of much larger illumination number.

We denote the closed Euclidean unit ball in $\Ren$ centered at the origin $o$ 
by $\Ben$, and its boundary, the sphere by $\Sen$.

\begin{thm}\label{thm:spiky}
 Let $1<D<1.116$ be given. Then for any sufficiently large dimension $n$, 
there is an $o$-symmetric convex body $K$ in $\Ren$, with illumination number 
\begin{equation}\label{eq:iBiggerD}
{\rm i}(K)=N(K,\inter K)\geq .05 D^n,
\end{equation}
for which 
\begin{equation}\label{eq:BKBcontainment}
\frac{1}{D}\Ben\subset K \subset \Ben.
\end{equation}
\end{thm}

We will use a probabilistic construction to find $K$. We are not aware of any 
lower bound for the Illumination Problem that was obtained by a probabilistic 
argument.

For a point $u\in\Sen$, and $0<\phi<\pi/2$, let $C(u,\phi)=\{v\in\Sen\st 
\sphericalangle(u,v)\leq\phi\}$ denote the spherical cap centered at $u$ of 
angular radius $\phi$. We denote the normalized Lebesgue measure (that is, the 
Haar probability measure on $\Sen$) of $C(u,\phi)$ by $\Omegan(\phi)$.

In Theorem~\ref{thm:almostballupper}, we give an upper bound for the 
illumination number for bodies close to the Euclidean ball. It follows from
\cite{BK09} but, for the sake of completeness, we will sketch a proof.

\begin{thm}\label{thm:almostballupper}
 Let $K$ be a convex body in $\Ren$ such that $\frac{1}{D}\Ben\subset K \subset 
\Ben$ for some $D>1$. Then the illumination number of $K$ is at most
\begin{equation}
{\rm i}(K)\leq
 \frac{n\ln n+n\ln\ln n+5n}{\Omegan(\alpha)},
\end{equation}
where $\alpha=\arcsin (1/D)$.
\end{thm}

By combining Theorem~\ref{thm:almostballupper} with the estimate 
\eqref{eq:BWkicsi} on $\Omegan$, one can see that \eqref{eq:iBiggerD} is 
asymptotically sharp, that is, the base $D$ cannot be improved.

Next, we consider an application of Theorem~\ref{thm:spiky}.
Let $K$ be an origin-symmetric convex body in $\Ren$, and denote its gauge 
function by $\|\cdot\|_K$ (that is, $\|p\|_K=\inf\{\lambda>0\st p\in \lambda 
K\}$, for any $p\in\Ren$). 
We use ${\rm vert}\,P$ to denote the set of vertices of the polytope $P$.
The \emph{illumination parameter}, introduced by K. Bezdek \cite{Be06}, is 
defined as
\[
{\rm ill}(K) =
\inf\left\{\sum_{p \in {\rm vert}\,P}
\|p\|_K \mid P \text{ a polytope such that } {\rm vert}\,P \text{
illuminates } K\right\}.
\]

The \emph{vertex index} of $K$, introduced by K. Bezdek and Litvak \cite{BL07}, 
is
\[
{\rm vein}(K) = \inf\left\{\sum_{p \in {\rm vert}\,P} \|p\|_
K \mid P \text{ a polytope such that } K \subseteq P\right\}.
\]
Clearly, ${\rm ill}(K)\geq {\rm vein}(K)$ for any centrally symmetric body $K$, 
and they are equal for smooth bodies. It is shown in \cite{BL07} that ${\rm 
vein}(\Ben)$ is of order $n^{3/2}$ (see also \cite{GL12}). 

By \eqref{eq:BKBcontainment}, for the body $K$ constructed in 
Theorem~\ref{thm:spiky} we 
have that ${\rm vein}(K)$ is of order $n^{3/2}$, while ${\rm ill}(K)\geq {\rm 
i}(K)$ is exponentially large. 

Thus, as an application of Theorem~\ref{thm:spiky}, we obtain that ${\rm 
ill}(K)$ 
and ${\rm vein}(K)$ are very far from each other for some $K$.

\section{Preliminaries}\label{sec:prelim}

We will rely heavily on the following estimates of $\Omeganp$ by B\"or\"oczky 
and 
Wintsche \cite{BW03}.

\begin{lem}[{B\"or\"oczky -- Wintsche \cite{BW03}}]\label{lem:BWcapsize} Let 
$0<\phi<\pi/2$.
\begin{eqnarray}
 \Omeganp(\phi)&>& \frac{\sin^n\phi}{\sqrt{2\pi(n+1)}}, 
\label{eq:BWnagy}
 \\
 \Omeganp(\phi)&<& \frac{\sin^n\phi}{\sqrt{2\pi n}\cos\phi} 
,\;\;\;\;\mbox{ if } \phi\leq \arccos \frac{1}{\sqrt{n+1}},
\label{eq:BWkicsi}
\\
 \Omeganp(t\phi)&<& t^n\Omeganp(\phi),\;\;\;\;\mbox{ if } 
1<t<\frac{\pi}{2\phi}. 
\label{eq:BWtszer}
\end{eqnarray}
\end{lem}

The following is known as Jordan's inequality:
\begin{equation}\label{eq:jordan}
\frac{2x}{\pi}\leq\sin x,\;\;\mbox{ for }\;\;x\in[0,\pi/2].
\end{equation}

\section{Construction of a Spiky Ball}\label{sec:spikyball}

We work in $\Renp$ instead of $\Ren$ to obtain slightly simpler formulas.
We describe a probabilistic construction of $K\subset\Renp$ 
which is close to the Euclidean ball and has a large illumination number. 
We use the standard notation $[N]$ for the set $\{1,\ldots,N\}$, and $\card{A}$ 
denotes the cardinality of a set $A$.

Let $N$ be a fixed positive integer, to be given later. We pick $N$ points, 
$X_1,\ldots,X_N$ independently and uniformly on the Euclidean unit sphere 
$\Senp$ of $\Renp$. Let 
\begin{equation}
 K=\conv\left(\{\pm X_i\st i\in [N]\}\cup \frac{1}{D}\Benp\right).
\end{equation}
Clearly, $K$ is $o$-symmetric and $\frac{1}{D}\Benp\subset K\subset \Benp$. We 
need to bound the illumination number of $K$ from below.
Let $\frac{\pi}{4}<\alpha<\frac{\pi}{2}$ be such that 
$\sin\alpha=1/D$.

\begin{figure}[tb]
    \centering
    \includegraphics[width=0.5\textwidth]{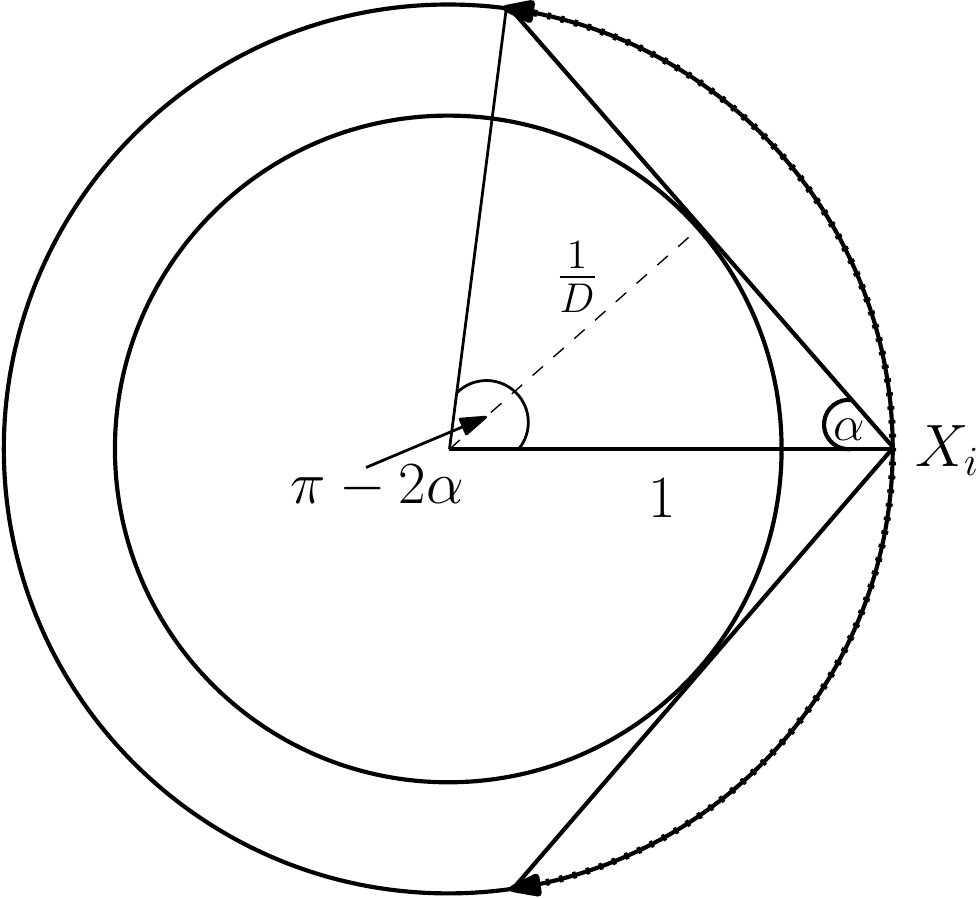}
    \caption[]{Event $E_1$: when $X_j$ falls on the dotted cap (the arc with 
arrows at its endpoints) or on its reflection about the origin.}
    \label{fig:bkp}
\end{figure}

We define two ``bad'' events, $E_1$ and $E_2$. Let $E_1$ be the event that 
there 
are $i\neq j \in [N]$ with $\sphericalangle(X_i,X_j)<\pi-2\alpha$ or 
$\sphericalangle(-X_i,X_j)<\pi-2\alpha$. We observe that if $E_1$ does not 
occur, then for all $i\in[N]$ we have

\fact{\label{fact:cap}
the set of directions (a subset of $\Senp$) that illuminate $K$ at $X_i$
is the spherical cap centered at $-X_i$ of spherical radius $\alpha$. 
}

We want to prove that, with non-zero probability, no point of $\Senp$ belongs 
to 
too many of these caps. Thus, to illuminate $K$ at each $X_i$, we will 
need many directions. 

Let $T\in\Ze^+$ be fixed, to be specified later.
Let $E_2$ be the event that there is a direction $u\in\Senp$ with 
$\card{C(u,\alpha)\cap \{\pm X_i\st i\in [N]\}}>T$.

Observe that if neither $E_1$ nor $E_2$ occur, then ${\rm i}(K)\geq 2N/T$.
However, it is difficult to bound the probability of $E_2$. Thus, we will 
replace $E_2$ by a ``more finite'' condition $E^{\prime}_2$ as follows.

We fix a $\delta>0$. We call a set $\Lambda\subset\Senp$ a \emph{$\delta$-net} 
(could also be called a metric $\delta$-net) if 
$\cup_{v\in\Lambda}C(v,\delta)=\Senp$, that is, if the caps of radius $\delta$ 
centered at the points of $\Lambda$ cover the sphere.
By \eqref{eq:BWnagy}, the measure of a cap of radius $\delta$ is larger than 
$\frac{\sin^n(\delta)}{3\sqrt{n}}$. Thus, Theorem~1 of \cite{Ro63} yields that, 
there is a covering of the sphere by at most $n^2/\sin^n(\delta)$ caps of 
radius $\delta$. That is, there is a $\delta$-net $\Lambda$ of size at most 
$\card{\Lambda}\leq n^2/ \sin^n(\delta)$.

Let $p=2\Omeganp(\alpha+\delta)$. Let $\Theta>1$ be fixed, and set 
$T=N\Theta p$.
We define the event $E^{\prime}_2$ as follows:
there is a direction $v\in\Lambda$ with $\card{C(v,\alpha+\delta)\cap \{\pm 
X_i\st i\in [N]\}}>N\Theta p$.
Clearly, if $E_2$ occurs, then so does $E^{\prime}_2$. Thus, we have
\begin{equation}\label{eq:implies}
 (\mbox{not}(E_1) \mbox{ and } \mbox{not}(E^{\prime}_2)) \mbox{ implies 
}{\rm i}(K)\geq 2/(\Theta p).
\end{equation}

Now, we need to set our parameters such that the event $(\mbox{not}(E_1) \mbox{ 
and } \mbox{not}(E^{\prime}_2))$ is of positive probability and $2/(\Theta p)$ 
is exponentially large in the dimension.

Clearly,
\begin{equation}\label{eq:pe1}
 \Pe(E_1)\leq N^2\Omeganp(\pi-2\alpha).
\end{equation}

Consider a fixed $v\in\Lambda$. When $X_i$ is picked randomly, the probability 
that $v$ is contained in $C(X_i,\alpha+\delta)$ or in $C(-X_i,\alpha+\delta)$ 
is $p$ (recall that $p=2\Omeganp(\alpha+\delta)$). Thus, 
the probability that $v$ is contained in more than 
$N\Theta p$ caps of the form $C(\pm X_i,\alpha+\delta)$ is $\Pe(\xi> 
 N\Theta p)$, where $\xi$ is a binomial random variable of distribution 
$\Binom(N,p)$. Thus,

\begin{equation}\label{eq:pe2}
 \Pe(E^{\prime}_2)\leq \frac{n^2}{\sin^n(\delta)}\Pe(\xi>N\Theta p) 
\;\;\mbox{ with } \xi\sim\Binom(N,p).
\end{equation}

By a Chernoff-type inequality, (cf. p. 64 of \cite{MU05}),
\begin{equation}\label{eq:cher}
\Pe(\xi>N\Theta p) <
2^{-N\Theta p},
\;\;\mbox{ for any }\;\;\Theta\geq 6.
\end{equation}

Consider the following three inequalities.
\begin{eqnarray}
 N&\leq& \left(\frac{1}{4\Omeganp(\pi-2\alpha)}\right)^{1/2},\label{eq:Nle}
\\
 \frac{n^2}{\sin^n\delta}2^{-\Theta Np}&\leq&\frac{1}{4},\label{eq:Nge}
\\
6&\leq& \Theta.\label{eq:thetap}
\end{eqnarray}

Combining \eqref{eq:implies}, \eqref{eq:pe1}, \eqref{eq:pe2} and 
\eqref{eq:cher}, we obtain the following.
If there are $N\in\Ze^+, \delta>0$ and $\Theta\geq 0$ (all depending on 
$n$) such that the three inequalities \eqref{eq:Nle},\eqref{eq:Nge} 
and \eqref{eq:thetap} 
hold, then there is a $K\subset\Renp$ $o$-symmetric convex 
body with ${\rm i}(K)\geq 2/(\Theta p)$, where $p=2\Omeganp(\alpha+\delta)$. In 
fact, in this case, our construction yields such a $K$ with probability at 
least $1/2$.

Now, \eqref{eq:Nge} holds if 
$\Theta N p>2n \log_2\frac{1}{\sin\delta}.$
Thus, an integer $N$ satisfying \eqref{eq:Nle} and \eqref{eq:Nge} exists if 
\begin{equation*}
 4n \log_2\frac{1}{\sin\delta} \leq 
 \Theta p \left(\frac{1}{4\Omeganp(\pi-2\alpha)}\right)^{1/2},
\end{equation*}
which we rewrite as
\begin{equation*}
 \frac{1}{\Theta p} \leq
 \frac{1}{8n (\Omeganp(\pi-2\alpha))^{1/2}\log_2\frac{1}{\sin\delta}}.
\end{equation*}

By \eqref{eq:jordan}, we can replace it
by the following stronger 
inequality:
\begin{equation}
 \label{eq:c11}
 \frac{1}{\Theta p} \leq
 \frac{1}{24n (\Omeganp(\pi-2\alpha))^{1/2}\log_2 (1/\delta)}.
\end{equation}

On the other hand, by substituting the value of $p$, we see that 
\eqref{eq:thetap} is equivalent to
\begin{equation}
 \frac{1}{\Theta p} \leq
 \frac{1}{12\Omeganp(\alpha+\delta)}.
 \label{eq:c2}
\end{equation}

Finally, let $\delta=\frac{\alpha}{n}$. 

Since $1<D=\frac{1}{\sin \alpha}<1.116$, we have that $1.11<\alpha<\pi/2$, and 
thus, $\sin^2(\alpha+\delta)>\sin(\pi-2\alpha)$. Now, by 
Lemma~\ref{lem:BWcapsize}, \eqref{eq:c2} is a stronger inequality than 
\eqref{eq:c11}. Thus, so far we have that if we can satisfy \eqref{eq:c2}, then 
the proof is complete.

By \eqref{eq:BWtszer}, we have that \eqref{eq:c2} holds, if
\begin{equation}
 \label{eq:cc}
 \frac{1}{\Theta p} \leq
 \frac{1}{36\Omeganp(\alpha)}.
\end{equation}
By \eqref{eq:BWkicsi}, it holds for $\frac{1}{\Theta p}=\frac{1}{36}D^n$. Since 
 ${\rm i}(K)\geq 2/(\Theta p)$, this finishes the proof of 
Theorem~\ref{thm:spiky}.

\begin{rem}
 The body $K$ is not a polytope. However, the construction can easily be 
modified to obtain a polytope. One simply replaces the ball of radius $1/D$ by 
a sufficiently dense finite subset $A$ of this ball in the definition of $K$ as 
follows:
$
 K=\conv(\{\pm X_i\st i\in [N]\}\cup A).
$
\end{rem}

\begin{proof}[Proof of Theorem~\ref{thm:almostballupper}]
Since $\frac{1}{D}\Ben\subset K\subset \Ben$, it follows that for any boundary 
point $b$ of $K$, the set of directions (as a subset of $\Sen$) that 
illuminate $K$ at $b$ contains an open spherical cap of radius 
$\alpha=\arcsin(1/D)$. Thus, any subset $A$ of $\Sen$ that pierces each 
such cap illuminates $K$. However, finding such $A$ is equivalent to finding a 
covering 
of $\Sen$ by caps of radius $\alpha$. Such a covering of the desired 
size exists by \cite{Ro63} (see also \cite{BW03}).
\end{proof}

\section*{Acknowledgement}
I am grateful for the conversations with Károly Bezdek, Gábor Fejes 
Tóth and J\'anos Pach. I also thank the referee whose comments made the 
exposition more clear.

\bibliographystyle{alpha}
\bibliography{biblio}

\end{document}